\documentclass[twoside]{article}
\usepackage[utf8]{inputenc}
\usepackage{amsmath}
\usepackage{amsfonts}
\usepackage{amssymb}
\usepackage{graphicx}
\usepackage{chngcntr}
\counterwithin*{equation}{section}
\usepackage{color}
\usepackage{setspace}
\usepackage[normalem]{ulem}
\usepackage{hyperref}
\usepackage{graphicx}
\usepackage{epstopdf}
\usepackage{float}
\usepackage{overpic}

\usepackage[left=2cm,right=2cm,top=2cm,bottom=2cm]{geometry}
\usepackage{multicol}
\usepackage{amsmath}
\usepackage{amsfonts}
\usepackage{amssymb}
\usepackage{graphicx}
\usepackage{color}
\usepackage[utf8]{inputenc}
\usepackage{amsfonts}
\usepackage{color}
\usepackage[normalem]{ulem}
\usepackage{nomencl}
\usepackage{algorithm,algorithmic,multirow}
\makenomenclature

\newenvironment{proof}{\noindent{\sc Proof.}}{\qed}
\newtheorem{theorem}{Theorem}[section]
\newtheorem{lemma}{Lemma}[section]
\newtheorem{cor}{Corollary}[section]
\newtheorem{rem}{Remark}[section]
\newtheorem{definition}{Definition}[section]
\newtheorem{prop}{Proposition}[section]
\newtheorem{uda}{Example}[section]
\newcommand{\qed}{$\blacksquare$}
\def\hindu{\arabic}

\renewcommand{\theequation}{\hindu{section}.\hindu{equation}}

\providecommand{\argmin}{\operatorname*{arg\ min}}

\newcommand\gattha[2]{\genfrac{}{}{0pt}{}{#1}{#2}}

\def\RR{{\mathbb R}}
\def\CC{{\mathbb C}}
\def\ZZ{{\mathbb Z}}

\def\C{{\mathcal C}}

\def\argmin{\mathop{\hbox{\textrm{arg min}}}}

\def\be{\begin{equation}}
\def\ee{\end{equation}}
\def\bea{\begin{eqnarray}}
\def\eea{\end{eqnarray}}

\def\disp{\displaystyle}

\def\donchitre#1#2{\vskip 6.5cm\noindent
\parbox[t]{1in}{\special{eps:#1.eps x=6.5cm y=5.5cm}}
\hbox to 7cm{}\parbox[t]{0.0cm}{\special{eps:#2.eps x=6.5cm y=5.5cm}}}

\def\gs{\gtrsim}
\def\ls{\lesssim}

\title{An Eigenfunction Approach to Conversion of the Laplace Transform of Point Masses on the Real Line to the Fourier Domain}
\author{Michael Mckenna$^1$,
 Hrushikesh~N.~Mhaskar$^2$
 \thanks{ Institute of Mathematical Sciences, Claremont Graduate University, Claremont, CA 91711 (United States).The research of HNM is supported in part by NSF grant DMS 2012355 and ONR grants N00014-23-1-2394, N00014-23-1-2790. The research of MM and RGS was funded in part by the National Institute on Aging, Intramural Research Program.
\textsf{email:} hrushikesh.mhaskar@cgu.edu.
}\ ,
Richard G. Spencer$^1$}
\begin{document}
\date{%
    $^1$ National Institute on Aging, National Institutes of Health, Laboratory of Clinical Investigation, Magnetic Resonance Imaging and Spectroscopy Section, Baltimore MD 21224\\%
    $^2$Institute of Mathematical Sciences, Claremont Graduate University, Claremont, CA 91711 U.S.A.\\[2ex]%
}

\maketitle
\begin{abstract}
Motivated by applications in magnetic resonance relaxometry, we consider the following problem: Given samples of a function $t\mapsto \sum_{k=1}^K A_k\exp(-t\lambda_k)$, where $K\ge 2$ is an integer, $A_k\in\mathbb{R}$, $\lambda_k>0$ for $k=1,\cdots, K$, determine $K$, $A_k$'s and $\lambda_k$'s.
Unlike the case in which the $\lambda_k$'s are purely imaginary, this problem is notoriously ill-posed. Our goal is to show that this problem can be transformed into an equivalent one in which the $\lambda_k$'s are replaced by $i\lambda_k$. 
We show that this may be accomplished by approximation in terms of Hermite functions, and using the fact that these functions are eigenfunctions of the Fourier transform. We present a preliminary numerical exploration of parameter extraction from this formalism, including the effect of noise. We do not claim to have eliminated the inherent ill-posedness of the original problem, as reflected in the numerical results.
\end{abstract}

\section{Introduction}\label{sec:intro}


Multiexponential decay signals, consisting of the superposition of two or more exponentials of the form:
\be\label{eq:observations}
F_0(t)=\sum_{k=1}^K A_k\exp(-t\lambda_k) +\epsilon_0(t), \qquad 0\le t\le T,
\ee
where $T>0$, $K$ is a positive integer, $A_k,\lambda_k>0$ for $k=1,\cdots, K$, and $\epsilon_0(t)$'s are i.i.d. random variables with mean $0$, arise in innumerable settings in the natural sciences and mathematics \cite{Istratov1999}, including optics \cite{Brubach2009}, astrophysics \cite{Rust2010}, food science \cite{Kirtil2016}, geophysics \cite{Sonnenschmidt1996}, and biology \cite{Brownstein1979,Oberhauser1998}.
The main problem is to determine the parameters $A_k$ and $\lambda_k$. The literature is less focused on determination of the number of components  $K$, although that may also be of substantial interest and can be addressed by a number of statistical methods. The value of $K$ is often assumed. 

Our main motivating example is magnetic resonance relaxometry, in which the proton nuclei of water are first excited with radio frequency pulses and then exhibit an exponentially decaying electromagnetic signal. These applications most often have $K=2$, although applications with $K \geq 3$ appear occasionally.
In general, biological tissue will contain multiple states of water defined by, among other factors, molecular mobility, often modeled as the co-existence of discrete compartments \cite{Graham1996}. In principle, these compartments can then be separately mapped based on the fact that they exhibit different MR transverse decay time constants, or $T_2$ values, as follows. A particularly important example is myelin mapping in the brain. Myelin forms an insulating sheath surrounding neuronal axons in the brain and spinal cord and is critical for the transmission of electrical signals \cite{Jahn2009}. Water within the myelin sheath itself exhibits limited mobility and therefore, according to MR theory, exhibits a more rapid electromagnetic signal decay as compared to the more mobile water. The more rapidly decying signal component is designated as the myelin water fraction (MWF). Thus, the MR signal is composed of sub-signals from both the MWF and the more mobile water, with these two contributions separately quantified by making use of their different relaxation time constants. MWF mapping in the human brain based on this principle was pioneered by Mackay, \cite{Mackay1994}, who focused on the related model of a discretized continuous range of tissue decay time constants. MWF mapping is applied  widely in the investigation of cognitive decline \cite{Bouhrara2020MWF,Gong2023LowerMyelin}, Alzheimer’s disease \cite{Dean2017,Kavroulakis2018}, neurodegenerative diseases such as multiple sclerosis \cite{Laule2004}, and in studies of normative aging that are characterized by or exhibit myelin loss \cite{Bouhrara2020CerebraBlood,
Bouhrara2021Maturation,Bouhrara2020QuantitativeMyelin,Bouhrara2021Gratio,Bouhrara2020MWF,Drenthen2019,Laule2004,Piredda2021,Vavasour2006}. We provide preliminary corresponding results below, in the section entitled "Numerical Results". 

In addition to widespread application, these techniques are still undergoing important refinement, due to the complexity of separating the signal generated by the MWF from the signal generated by the more mobile water. The difficulty in this and related multiexponential analyses may be understood in the following way.  It is well-known that the solution of the Fredholm equation of the first kind (FEFK) 

\be\label{eq:Fredholm}
F(t) = \int_{a}^{b}K(t,\tau) d\mu(\tau)
\ee
is an ill-posed inverse problem, with the determination of the unknown measure $\mu$  being highly sensitive to noise. When the kernel function is the exponential, $K(t, \tau) = exp(-t\tau)$,  parameter extraction is a special case of the notoriously ill-posed inverse Laplace transform (ILT)\cite{Celik2013}.  If the support of the unknown distribution function is the sum of point masses, we recover the multiexponential decay problem \eqref{eq:observations}, which inherits its ill-posedness from the FEFK.   
In the discretized case, these difficulties are highlighted by the condition number of the matrix defining the discretization of the FEFK, which can become enormous with fine discretization, even for only two monoexponential  components \cite{Bonny2020,Spencer2020TutorialIntroInverse}. This effect of increasingly finer discretization may be understood by the fact that it brings the discretized problem closer to the ill-posed native integral equation of the FEFK.  

This behavior is in marked contrast to that of the Fourier transform; as a unitary transform, with appropriate normalization, the FT enjoys the optimal condition number of unity; that is, $\|\mathfrak{F}^*\|\|\mathfrak{F}\| = 1$, where $\|\cdot\|$ is the $L^2$ operator norm \cite{Spencer2020TutorialIntroInverse}. 
Given this, the transformation of the ILT to an equivalent Fourier problem is an appealing approach, as was first (to our knowledge) described by Gardner. This allowed for the visualization of decay constants as peaks on a frequency spectrum \cite{Gardner1959}. This work was later built upon by Provencher through use of a of tapered window function to reduce sidelobes in the frequency domain defining the rates $\lambda_k$ \cite{Provencher1976}.  

Here we present a further development in this area, with the goal of transforming the problem, stated most naturally in terms of parameter estimation from the inverse Laplace transform of point masses, to their inverse Fourier transform. To do this, we exploit the properties of the orthonormalized Hermite functions as eigenfunctions of the Fourier transform.  Our approach is in some sense reminiscent of previous work on the expansion of multiexponential signals by orthogonal polynomials \cite{Holmstrom2002}, including an earlier approach for inversion of the Laplace transform through the use of the Jacobi and Laguerre polynomials. \cite{Piessens1972IverseLaplaceNewMethod}. Furthermore, unlike many previous mathematical treatments of multi-component multiexponential decay signals, we incorporate experimental noise into our analysis.

The organization of this paper is as follows. 
We describe some related work in Section~\ref{sec:related}. 
Section~\ref{sec:notation} presents definitions and notation. Section~\ref{sec:idea} introduces our main idea, which is expansion of the multiexponential function in terms of Hermite polynomials and transformation to Fourier space.
 Section~\ref{sec:quadrature} presents a general theory for computation of the coefficients in the expansion of Hermite polynomials from discrete values of $F_0$. 
 We describe the process of least square approximation and related quantities in Section~\ref{sec:leastsq}.
 Based on these considerations, Section~\ref{sec:main} describes our main theorems concerning the errors in the various approximations, including errors due to noise. The theorems in Section~\ref{sec:main} are proved in Section~\ref{sec:proofs}. The procedure for obtaining $\lambda_k$'s from the Fourier transform is based on our earlier work, and is summarized in the Appendix.
 A numerical implementation is given in Section~\ref{sec:Michael}.

\section{Related work}\label{sec:related}

Many methods have been proposed, developed, and implemented in the past 200 years for exponential analysis \cite{Bi2022Stabilization,Brubach2009, Harman1973, Holmstrom2002, Istratov1999,Piessens1975Bibliography,
Prony1795,Rust2010}.
 These methods range from the historical Prony method, still used in modified form, to neural network approaches \cite{Heliot2021,Parasram2021, Rozowski2022, Xiao2021}.
Among the recent methods are various subspace methods such as ESPRIT and ESPIRA \cite{plonka2018numerical, Potts2010, cuyt2024multiscale}. 
 Although these methods work in theory for complex exponents, they are mainly used with purely imaginary exponents and are not effective for real exponents.
 Extensive review material may be found in \cite{Holmstrom2002, Istratov1999}.  
There is some related work e.g., \cite{hackbusch2019computation, derevianko2023differential} which deals with the approximation of functions using real exponential sums.
One could think of the problem as one of extracting the parameters in \eqref{eq:observations} when the values of $F$ are known deterministically but only approximately. 
Another work of particular interest introducing an alternative and highly promising approach is that of Derevianko and Plonka\cite{derevianko2022exact}. 
The authors are motivated by the fact that
$$
\int_\RR \left\{\sum_{k=1}^K A_k\exp(-\lambda_k t)\right\}\exp(-i\omega t)dt = \sum_{k=1}^K \frac{A_k}{i(\omega -i\lambda_k)},
$$
so that finding the poles of a rational approximation to $F_0$ in \eqref{eq:observations} (extended as a periodic function so that a discrete Fourier transform can be used instead) leads to the determination of the parameters.
The method assumes that all the $A_k$'s are positive, and is shown to be accurate for parameter recovery in the absence of noise; the authors further claim that the parameter recover is stable under numerical perturbations.
Effectively, the Fourier transform converts $F_0$ (without noise) to the so-called Stieltjes transform of a measure that associates the mass $A_k$ with $i\lambda_k$. 
It is well known that the denominator of the Pad\'e approximation of this transform is the Prony polynomial, which is an orthogonal polynomial on the unit circle of degree $K$.
Our previous paper \cite{singdet} is similar in spirit to this direction, and works without the assumption that the $A_k$'s are positive, but restricted to purely imaginary exponents.
Finally, we note that the paper \cite{bspaper} deals with the case when the exponents are complex numbers with small real parts and distinct imaginary parts.

 Given theoretical bounds on the ability to define these parameters as a function of signal-to-noise ratio (SNR), measurement intervals and total duration, and number of underlying exponentials
 \cite{Bertero1982,Bertero1984,Bertero1985,Bertero1984SingularValue}, development of further methodology represents attempts to extract the maximum information from a given signal, or to reformulate the problem to circumvent these limitations \cite{Bi2022Stabilization,Hampton2023}.

  We emphasize that in this work, we are analyzing the expression \eqref{eq:observations}, which has been applied to both myelin quantification and other tissue characterization studies in MRI \cite{Bouhrara2015, Bouhrara2015bayesian, Charles2006,Mackay2006, Mackay1994,Sharafi2017, Shinmoto2009}. A related body of published work in the magnetic resonance literature instead aims to recover a spectrum of $T_2$ values; the underlying physics is similar for the two approaches, although the mathematical analysis differs between them. 

\section{Notation}\label{sec:notation}

The orthonormalized Hermite functions are defined (cf. \cite[Formula~(5.5.3)]{szego}) by
\be\label{eq:hermitefndef}
\psi_k(x)=(\pi^{1/2}2^k k!)^{-1/2}(-1)^k\exp(-x^2/2)\left(\frac{d}{dx}\right)^k\exp(-x^2), \qquad k\in\ZZ_+.
\ee
For $n>0$, we denote
$$
\Pi_n=\mathsf{span}\{\psi_k : k<n\}.
$$
Each element of $\Pi_n$ is a function of the form $x\mapsto P(x)\exp(-x^2/2)$ for some polynomial $P$ of degree $<n$. Such a function will be called a \emph{weighted polynomial of degree $<n$}.
We note that the notation does not require $n$ to be an integer; we interpret $\Pi_n=\Pi_{\lfloor n\rfloor}$ if $n$ is not an integer. This simplifies some notation later.

The Fourier transform of a function $f :\RR\to\RR$ is defined formally by
\be\label{eq:fourtransdef}
\mathfrak{F}(f)(\omega)=\frac{1}{\sqrt{2\pi}}\int_\RR f(t)\exp(-i\omega t)dt.
\ee

It is known  (cf. \cite[Theorem~4.3.2]{mhasbk}; the notation is different) that
\be\label{eq:hermitefour}
\mathfrak{F}(\psi_k)(\omega)=(-i)^k \psi_k(\omega), \qquad k\in\ZZ_+, \ \omega\in\RR.
\ee


Next, we introduce some concepts and notation related to various measures which we will use in this paper.

Rather than writing sums of the form $\sum_{j=1}^m w_jg(y_j)$, it is convenient to use the Stieltjes' integral notation. 
Let $\delta_x$ be the Dirac delta at $x$. 
As a functional (or measure), this is expressed as $\int gd\delta_x=g(x)$. 
We define a (possibly complex) measure $\nu=\sum_{j=1}^m w_j\delta_{y_j}$ (described as a measure that associates the mass $w_j$ with each $y_j$) and write $\int g(y)d\nu(y)$ or even shorter, $\int gd\nu$, instead of $\sum_{j=1}^m w_jg(y_j)$. 
The \emph{support} of $\nu$ (denoted by $\mathsf{supp}(\nu)$) is $\{y_j\}_{j=1}^m$, and $\nu$ is said to be \emph{supported on} any set that contains the support.
The quantity $\sum_{j=1}^m |w_j|$ is the \emph{total variation} of $\nu$, denoted by $\|\nu\|_{TV}$.
There is a measure associated with $\nu$ called the total variation measure, denoted by $|\nu|$. For example, for the above measure $\nu$,
$$
\int gd|\nu|=\sum_{j=1}^m |w_j|g(y_j).
$$
The Laplace transform $\mathcal{L}(\nu)$ and the Fourier transform $\mathfrak{F}(\nu)$ for $\nu=\sum_{j=1}^m w_j\delta_{y_j}$ are defined by
\be\label{eq:measure_transforms}
\mathcal{L}(\nu)(t)=\int \exp(-tu)d\nu(u)=\sum_{j=1}^m w_j\exp(-ty_j), \quad \mathfrak{F}(\nu)(\omega)=\sum_{j=1}^m w_j\exp(-i\omega y_j).
\ee
If $\nu$ is any possibly signed measure, and $A\subseteq\RR$ is $\nu$-measurable, we will write formally
\be\label{eq:lpnorm}
\|f\|_{\nu;p, A}=
\begin{cases}
\disp\left\{\int_A |f|^pd|\nu|\right\}^{1/p}, &\mbox{ if $1\le p<\infty$,}\\[1ex]
\sup_{t\in \mathsf{supp}(\nu)\cap A}|f(t)|, &\mbox{ if $p=\infty$.}
\end{cases}
\ee
The space $L^p(\nu)$ comprises measurable functions $f$ for which $\|f\|_{\nu;p}<\infty$. 
If $A=\RR$, we will omit the mention of $A$ from the notation. 
When $\nu$ is the Lebesgue measure on $\RR$, we will omit its mention and write $\|f\|_p$ in place of $\|f\|_{\nu;p}$, $L^p$ instead of $L^p(\nu)$, etc.
As usual, we identify functions which are equal $|\nu|$-almost everywhere.

For $f\in L^p$, we  define the degree of approximation of $f$ by
\be\label{eq:degapprox}
E_{p,n}(f)=\min_{P\in\Pi_n}\|f-P\|_p.
\ee

\textbf{Constant convention:}\\

\textit{In the sequel, $c, c_1,\cdots$ will denote generic positive constants, whose values may be different at different occurrences, even within a single formula. 
The notation $A\ls B$ means $A\le cB$, $A\gs B$ means $B\ls A$, $A\sim B$ means $A\ls B\ls A$. 
For example, in a formula of the form
$$
A_n\ls n^c\exp(-c_1n)\ls \exp(-cn),
$$
there are implied constants in $\ls$, and the constants $c, c_1$ in the middle term are different from the constant $c$ in the last term.
}

Finally, we recall a fundamental result from \cite[Theorem~6.2.4, Theorem~6.1.6]{mhasbk}.

\begin{prop}\label{prop:mrs}
For any $\delta>0$, $n>0$, $1\le p\le \infty$, and $P\in\Pi_n$,
\be\label{eq:mrsineq}
\|P\|_{p,\RR\setminus [-\sqrt{2n}(1+\delta), \sqrt{2n}(1+\delta)]}\le \exp(-cn)\|P\|_{p, [-\sqrt{2n}(1+\delta), \sqrt{2n}(1+\delta)]}.
\ee
In particular,
\be\label{eq:mrs}
\|P\|_p\sim \|P\|_{[-2\sqrt{n}, 2\sqrt{n}]}.
\ee
\end{prop}
\section{Fundamental Idea}\label{sec:idea}

To explain our main idea, we first define 
a noiseless signal by (cf. \eqref{eq:observations})
\be\label{eq:idealsignal}
F(t)=\sum_{k=1}^K a_k\exp(-t\lambda_k), \qquad t\in\RR.
\ee
With the measure
\be\label{eq:basicmeasure}
\mu^*(t)=\sum_{k=1}^K a_k\delta_{\lambda_k},
\ee
where $\delta_\lambda$ denotes the Dirac delta supported at $\lambda$, we observe that $F$ is the Laplace transform of $\mu^*$. 
The fundamental idea is to convert this Laplace transform of $\mu^*$ into the Fourier transform of another measure $\tilde{\mu}$ supported on $\{\lambda_k\}$ as in \eqref{eq:fouriermu} below.
Towards this end, we introduce another function
\be\label{eq:completesq}
f(t)=F(t)\exp(-t^2/2)=\sum_{k=1}^K 
a_k\exp(\lambda_k^2/2)\exp(-(t+\lambda_k)^2/2).
\ee
and observe the fact that $\mathfrak{F}(\exp(-\cdot^2/2))(\omega)=\exp(-\omega^2/2)$.  
We obtain: 
\be\label{eq:prefundarelation}
\mathfrak{F}(f)(-\omega)=\exp(-\omega^2/2)\sum_{k=1}^K a_k\exp(\lambda_k^2/2)\exp(-i\omega\lambda_k).
\ee
With 
\be\label{eq:fouriermu}
\tilde{\mu}(u)=\sum_{k=1}^K a_k\exp(\lambda_k^2/2)\delta_{\lambda_k},
\ee
we observe 
\be\label{eq:fundarelation}
\mathfrak{F}(\tilde{\mu})(\omega)=\exp(\omega^2/2)\mathfrak{F}(f)(-\omega).
\ee

We will describe in the appendix a method to estimate the $\lambda_k$'s from samples of $\mathfrak{F}(\tilde{\mu})(\omega)$ using our theory of localized trigonometric polynomial kernels \cite{bspaper,singdet,trigwave, loctrigwave}.

In order to compute $\mathfrak{F}(\tilde{\mu})(\omega)$, our basic idea is to  approximate $f$ in terms of the Hermite functions, and use \eqref{eq:hermitefour}.

There are three technical barriers here: one is that the signal $F_0$ as defined in \eqref{eq:observations} is defined only on a finite interval.
The second is that we know only finitely many values of $F_0$, not a formula for $F_0$.
Finally, there is noise.

We will address the first barrier in this section, the second in Section~\ref{sec:quadrature}, and study the effect of noise in Section~\ref{sec:main}.

To remedy the barrier of knowing the signal only a finite interval, we will shift the signal to the left, so that the resulting interval on which the observations lie can contain all the quadrature points which are needed to evaluate the integrals from the values of the functions.
To carry out this program, we introduce a tunable parameter $L>0$,  write $R=T-L$, and define
\be\label{eq:orignalsignal}
\begin{aligned}
F_{\mbox{shifted}}(t)&=F_0(t+L)=\begin{cases}
\sum_{k=1}^K A_k\exp(-L\lambda_k)\exp(-t\lambda_k)+\epsilon(t), &\mbox{ if $-L\le t\le R$,}\\
0 &\mbox{ otherwise,}
\end{cases}\\
&=\begin{cases}
F(t)+\epsilon(t), &\mbox{ if $-L\le t\le R$,}\\
0 &\mbox{ otherwise,}
\end{cases}
\end{aligned}
\ee
where $F$ is defined as in \eqref{eq:idealsignal} with
\be\label{eq:akdef}
a_k=A_k\exp(-L\lambda_k), \qquad \epsilon(t)=\epsilon_0(t+L).
\ee
We observe that with $f$ defined in \eqref{eq:completesq} and 
\be\label{eq:weightednoise}
E(t)=\begin{cases}
\epsilon_0(t+L)\exp(-t^2/2), &\mbox{if $t\in [-L,R]$,}\\
0 &\mbox{otherwise},
\end{cases}
\ee 
we have 
\be\label{eq:originalsignal_final}
F_{\mbox{shifted}}(t)\exp(-t^2/2)=\begin{cases}
f(t)+E(t), &\mbox{ if $-L\le t\le R$,}\\
0 &\mbox{ otherwise}.
\end{cases}
\ee
To summarize, starting from the signal \eqref{eq:observations}, we define $a_k$'s by \eqref{eq:akdef}, $F$ by \eqref{eq:idealsignal}, $f$ by \eqref{eq:completesq}, and modify the noise $\epsilon_0$ in \eqref{eq:observations} as per \eqref{eq:weightednoise}.
In the rest of the paper, we will treat $f$ as the signal, and $E$ as the noise.

\section{Discretization formulas}\label{sec:quadrature}

Another issue that must be addressed is the fact that in the context of signal analysis, we know the values of $f$ only at finitely many points on $[-L,R]$; these are the points at which the signal is sampled.  Therefore, the expansion of $f$ must be calculated in terms of Hermite functions using only these values. 
 This motivates our use of the Gauss quadrature formula to compute the coefficients 
 $\int_\RR f(t)\psi_k(t)dt$. 
 However, it is a delicate matter to compute the quadrature nodes and weights \cite{townsend2015fast,trefethen2022exactness}.
 Moreover,  this requires the knowledge of $f$ at exactly the Gauss quadrature nodes, which is not the case in practice. The existence of 
 quadrature formulas based on arbitrary nodes are known to exist \cite{mohapatrapap}, but
it will be seen that for our purposes we do not need quadrature formulas exact for integrating products of Hermite functions of some order. 
Rather, it is sufficient to approximate $f$ using 
least squares approximation, provided the points at which $f$ is evaluated are sufficiently dense in $[-L,R]$.
This is encapsulated in the following definition.
\begin{definition}\label{def:mzmeasuredef}
A sequence $\{\nu_m\}_{m=1}^\infty$  of possibly signed measures is called a \textbf{strongly Marcinkiewicz Zygmund sequence} if 
each of the following two conditions is satisfied.
\be\label{eq:strongsmz}
\int_\RR |P(t)|^2d|\nu_m|(t) \sim \int_\RR |P(t)|^2dt, \qquad P\in\Pi_m.
\ee
\be\label{eq:mztotalvariation}
\|\nu_m\|_{TV} \sim m^{1/2}.
\ee
By an abuse of notation, we may say that $\nu_m$ is a strongly Marcinkiewicz Zygmund measure of order $m$, and write $\nu_m \in SMZ(m)$.
\end{definition}

\begin{uda}\label{uda:trivial}
{\rm
Obviously, the Lebesgue measure on $\RR$ satisfies \eqref{eq:strongsmz} for every $m\ge 1$, but not the condition \eqref{eq:mztotalvariation}. 
However, in view of Proposition~\ref{prop:mrs}, the Lebesgue measure restricted to $[-2\sqrt{m},2\sqrt{m}]$ belongs to $SMZ(m)$.
\qed}
\end{uda}
\begin{uda}\label{uda:quadrature}
{\rm
Let $x_{m,m}<\cdots<x_{1,m}$ be the zeros of $\psi_m$, and $\lambda_{k,m}$ be the corresponding weights in the Gauss quadrature formula:
\be\label{eq:gauss}
\sum_{k=1}^m \lambda_{k,m}\exp(x_{k,m}^2/2)P(x_{k,m})=\int_\RR P(t)\exp(-t^2/2)dt, \qquad P\in \Pi_{2m}.
\ee
Writing $w_{k,m}=\lambda_{k,m}\exp(x_{k,m}^2)$, it follows that
$$
\sum_{k=1}^m w_{k,m}|P(x_{k,m})|^2 =\int |P(t)|^2dt, \qquad P\in\Pi_m.
$$
Let $\nu_m$ be the measure that associates the mass $w_{k,m}$ with each $x_{k,m}$. 
Clearly, $\nu_m$ satisfies \eqref{eq:strongsmz}. 
In view of \cite[Theorem~8.2.7]{mhasbk} (used with $p=2$, $b=0$), $\nu_m$ satisfies \eqref{eq:mztotalvariation} as well, so that $\nu_m\in SMZ(m)$.
\qed
}
\end{uda}

\begin{uda}\label{uda:batenkov}
{\rm
There exist $c, c_1$ with the following property \cite{batenkov2018stable}: Let $\{x_{k,m}\}_{k=1}^m$ be any real numbers, $x_{m,m}<\cdots<x_{1,m}$, such that
\be\label{eq:batenkov1}
[x_{m,m}, x_{1,m}]\subseteq [-\sqrt{2m}(1+cm^{-2/3}), \sqrt{2m}(1+cm^{-2/3})],
\ee
and
\be\label{eq:batenkovsep}
|x_{k,m}-x_{k+1,m}|\le c_1m^{-1/2}.
\ee
Let $\nu_m$ be the measure that associates the mass $x_{k,m}-x_{k+1,m}$ with $x_{k,m}$.
Then for $1\le p<\infty$,
\be\label{eq:batenkovmz}
\left|\|P\|_{\nu_m;p}^p -\|P\|_p^p\right| \le (1/4)\|P\|_p^p, \qquad P\in \Pi_m.
\ee
Clearly, $\nu_m\in SMZ(m)$.
\qed}
\end{uda}

\section{Least squares approximation}\label{sec:leastsq}

Given integers $m\ge n\ge 1$, a measure $\nu\in SMZ(m)$, and $f \in C_0(\RR)$, we can define the least square approximation to $f$ from $\Pi_n$ by
\be\label{eq:leastsqop}
S_n(\nu; f)=\argmin_{P\in\Pi_n}\|f-P\|_{\nu;2}^2=\argmin_{P\in \Pi_n}\int |f(x)-P(x)|^2d|\nu|(x).
\ee 
\emph{In the sequel, we will assume that $\nu$ is a positive measure, so that $\nu=|\nu|$.}
\begin{rem}\label{rem:relevance}
{\rm
We note that $S_n(\nu;f)$ is defined entirely in terms of the restriction of $f$ to $\mathsf{supp}(\nu)$. Thus, with the notation as in Section~\ref{sec:notation}, and $f$ as defined in \eqref{eq:completesq}, if $\mathsf{supp}(\nu)\subset [-L,R]$, then $S_n(\nu;f)=S_n(\nu;f|_{[-L,R]})$.
\qed
}
\end{rem}
%

%
%
%

We will need an explicit expression for $S_n(\nu;f)$ in the form $\int f(y)K_n(\nu;x,y)d\nu(y)$ for a reproducing kernel $K_n(\nu;x,y)$ (cf. \eqref{eq:darboux}).

To obtain an expression for this kernel, 
we define the Gram matrix $\mathbf{G}$ by
\be\label{eq:grammatrix}
\mathbf{G}_{\ell,k}=\int \psi_\ell(x)\psi_k(x)d\nu(x), \qquad \ell, k=0,\cdots, m-1.
\ee
A quadratic form for $\mathbf{G}$ is
\be\label{eq:gquadform}
\sum_{\ell, k}d_\ell \overline{d_k}\mathbf{G}_{\ell, k} =\int \left|\sum_j d_j \psi_j\right|^2d\nu.
\ee
Since
$$
\int \left|\sum_j d_j \psi_j(x)\right|^2dx=\sum_j |d_j|^2,
$$
the estimate \eqref{eq:strongsmz} and the Raleigh-Riesz principle shows that $\mathbf{G}$ is positive definite, and all of its singular values are bounded from above and below by constants independent of $m$. 
In particular, $\mathbf{G}$ is invertible and positive definite, with all of the singular values of $\mathbf{G}^{-1}$ also bounded from above and below by constants independent of $m$. 
Thus, there exists a lower triangular matrix $\mathbf{L}$ so that the Cholesky decomposition
\be\label{eq:cholesky}
\mathbf{G}^{-1}=\mathbf{L}\mathbf{L}^*
\ee
holds.
Moreover, for all $\{d_\ell\}\subset \CC$,
\be\label{eq:invGnorm}
\sum_{\ell, k}d_\ell \overline{d_k}\mathbf{G}_{\ell, k}  \sim \sum_{\ell}|d_\ell|^2.
\ee
Let
\be\label{eq:nuorthopoly}
\tilde{p}_j =\sum_{\ell=0}^{j}\mathbf{L}_{\ell,j}\psi_\ell,\qquad j=0,\cdots,m-1.
\ee
Using the definition of the Gram matrix and the Cholesky decomposition in \eqref{eq:cholesky}, it is easy to verify that
\be\label{eq:nuorthorel}
\int \tilde{p}_j(x)\tilde{p}_k(x)d\nu(x)=\delta_{j,k}, \qquad j, k=0,\cdots,m-1.
\ee

Next, we define the reproducing kernel
\be\label{eq:darboux}
K_n(\nu;x,y)=\sum_{j=0}^{n-1} \tilde{p}_j(x)\tilde{p}_j(y)=\sum_{\ell,\ell'=0}^{n-1} \mathbf{G}^{-1}_{\ell,\ell'}\psi_\ell(x)\psi_{\ell'}(y), \qquad x,y\in\RR.
\ee

The Fourier transform of $K_n(\nu;\cdot, y)$ is given by
\be\label{eq:fourdarboux}
\widehat{K}_n(\nu;x,y)=\mathfrak{F}(K_n(\nu;\cdot,y))(x)=\sum_{j=0}^{n-1} \mathfrak{F}(\tilde{p}_j)(x)\tilde{p}_j(y)
\ee
It is not difficult to verify that for $x\in\RR$,
\be\label{eq:leastsqexpression}
S_n(\nu;f)(x)=\int f(y)K_n(\nu;x,y)d\nu(y), \qquad \mathfrak{F}(S_n(\nu;f))(\omega)=\int f(y)\widehat{K}_n(\nu;\omega,y)d\nu(y).
\ee

Computationally, the procedure to find $S_n(\nu;f)$ is as follows.
We define
\be
\label{eq:coeffIntegral}
 \hat{f}(\nu;\ell)=\int f(y)\psi_\ell(y)d\nu(y), \qquad \ell=0,\cdots, n-1,
\ee
and solve  the linear system of equations
\be
\label{eq:FindCoeff_SoE}
\sum_{k=0}^{n-1}\mathbf{G}_{\ell,k}d_k(f)=\hat{f}(\nu;\ell), \qquad \ell=0,\cdots, n-1.
\ee
for $d_k(f)$'s.
Then $S_n(\nu;f)$ is given by
$S_n(\nu; f)=\sum_{k=0}^{n-1} d_k(f)\psi_k$.

\begin{rem}\label{rem:fourpartial}
{\rm
In the case when $\nu$ is the Lebesgue measure on $\RR$, $S_n(\nu;f)$ is the partial sum of the Hermite expansion.
The corresponding reproducing kernel is given by
\be\label{eq:hermite_reprod_kern}
K_n(x,y)=\sum_{k=0}^{n-1}\psi_k(x)\psi_k(y), \qquad x, y\in\RR.
\ee
\qed}
\end{rem}

\section{Main theorems}\label{sec:main}

As remarked earlier, if $\nu\in SMZ(m)$ is supported on $[-L,R]$, then 
\be\label{eq:mainintro}
S_n(\nu;F_{\mbox{shifted}}\exp(-\cdot^2/2))=S_n(\nu;f)+S_n(\nu;E),
\ee 
can be calculated using observed values of $F_0$.

Our first theorem is a general theorem, giving estimates on the error in approximating $\mathfrak{F}(f)$ by the Fourier transform of the least square approximation to $f$; i.e., $\mathfrak{F}(S_n(\nu;f))$.
The first estimate \eqref{eq:fourdegapprox} has a desired weaker dependence on $n$, but the estimate is in comparison with the best polynomial approximation of $f$ from $\Pi_{n/2}$ rather than from  $\Pi_{n}$. 
This is useful when this degree of approximation decreases polynomially with $n$. 
The second estimate \eqref{eq:fourprojapprox} compares the error in terms of the best approximation from $\Pi_n$ itself (the same degree as $S_n(\nu;f))$.
This exhibits a more rapid rate of convergence as a function of $n$ when this degree of approximation decreases faster than any polynomial in $1/n$, as is the case for \eqref{eq:completesq} (cf. \eqref{eq:errest}), our main concern in this paper.
\begin{theorem}
\label{theo:fourdegapprox}
Let $f\in C_0\cap L^1$, $m\ge n\ge1$ be integers and, $\nu\in SMZ(m)$ be a positive measure. 
Then
\be\label{eq:fourdegapprox}
\|\mathfrak{F}(f)-\mathfrak{F}(S_n(\nu;f))\|_\infty \ls \|f-S_n(\nu;f)\|_1\ls E_{1,n/2}(f)+nE_{\infty, n/2}(f).
\ee
Alternately,
\be\label{eq:fourprojapprox}
\|\mathfrak{F}(f)-\mathfrak{F}(S_n(\nu;f))\|_\infty \ls \|f-S_n(\nu;f)\|_1\ls n^{1/2}E_{1,n}(f)+n^{3/2}E_{\infty, n}(f).
\ee
\end{theorem}

When $f$  is a discrete Laplace transform as in \eqref{eq:completesq}, we need to compare the desired Fourier transform in \eqref{eq:fundarelation} with the Fourier transform of $S_n(\nu;f)$. Here, $f$ being an entire function of exponential type, the estimate \eqref{eq:fourprojapprox} exhibits a more rapid rate of convergence with respect to $n$. 

Although the following theorem is just a corollary of Theorem~\ref{theo:fourdegapprox}, we list it as a separate theorem in view of its relation with the central problem of this paper.
\begin{theorem}
\label{theo:noiselessfour}
Let $f$ be as defined in \eqref{eq:completesq}, $m\ge n\ge 1$ be  integers, $\nu\in SMZ(m)$, and $\mathsf{supp}(\nu)\subseteq [-L,R]$. Let $\Lambda=\max_{1\le \ell \le K}\lambda_\ell$. Then
\be\label{eq:noiselessfour}
\|\mathfrak{F}(f)-\mathfrak{F}(S_n(\nu;f))\|_\infty \ls \|f-S_n(\nu;f)\|_1\ls \frac{n^{3/4}}{\sqrt{n!}}(\Lambda/\sqrt{2})^n\exp(\Lambda^2/4)\sum_{\ell=1}^K |a_\ell|.
\ee
\end{theorem}

Next, we turn our attention to the noise part, $S_n(\nu;E)$.
We recall the notation \eqref{eq:akdef} and \eqref{eq:weightednoise} that relates the quantity $E$ below with the noise $\epsilon_0$ in \eqref{eq:observations}.
\begin{theorem}
\label{theo:noisetheo}
Let $m\ge n\ge 1$ be integers, $\nu\in SMZ(m)$, $\mathsf{supp}(\nu)\subseteq [-L,R]$.
With the constant $C$ as in Lemma~\ref{lemma:hermitelemma}(b), 
let 
\be\label{eq:degcond}
C\sqrt{m}\ge 2n \ge c\log m
\ee
for a suitable choice of $c$. 
We assume that  the expected value of each $E(y)$ is $0$, and there exists $S>0$ such that
\be\label{eq:noisecond}
|E(y)|\le S, \qquad y\in\RR.
\ee
Then for any $\delta\in (0,1)$, we have with probability $>1-\delta$
\be\label{eq:singleEest}
\left\|\int E(y)\hat{K}_n(\nu;w,y)d\nu(y)\right\|_\infty \ls S\left(\frac{\log m}{m}\right)^{1/8}\sqrt{\log\left(\frac{\log m}{\delta}\right)}.
\ee  
\end{theorem}
\begin{rem}\label{rem:noiserem}
{\rm
We can relax the assumption \eqref{eq:noisecond} to assume that $E$ is sub-Gaussian. 
However, this just means that \eqref{eq:noisecond} holds with a high probability, and complicates the notation without adding further insight.
\qed
}
\end{rem}

The following corollary is a simple consequence of Theorem~\ref{theo:noisetheo} and the H\"offding's inequality Lemma~\ref{lemma:hoeffding}.

\begin{cor}\label{cor:meanEest}
With the conditions as in Theorem~\ref{theo:noisetheo}, if we repeat the calculations with $M$ instances of $E(y)$: $E_1(y),\cdots, E_M(y)$, then with probability $>1-\delta$
\be\label{eq:meanEest}
\left\|\frac{1}{M}\sum_{j=1}^M \int E_j(y)\hat{K}_n(w,y)d\nu(y)\right\|_\infty \ls SM^{-1/2}\left(\frac{\log m}{m}\right)^{1/8}\sqrt{\log\left(\frac{M\log m}{\delta}\right)}.
\ee
\end{cor}

\section{Numerical Results}\label{sec:Michael}
We present preliminary numerical results to establish the practicality of the above analysis. We denote the procedure for determining the $\lambda_j$'s in \eqref{eq:observations}  as ``Laplace to Fourier (L2F).'' Our examples will illustrate results for several biexponential signals, all with $A_1$=0.5, $A_2$=0.5. Values of $T_{21}$=10ms and $T_{22}$=50ms, with corresponding $\lambda_1=1/T_{21}$ and $\lambda_2=1/T_{22}$, reflect a relatively easy parameter estimation problem, since the ratio of decay constants is large. Values of $T_{21}$=40ms and $T_{22}$=60ms represent a much more difficult problem, while $T_{21}$=50ms and $T_{22}$=60ms is still harder. All parameters may of course be varied to explore behavior in different regions of parameter space. 

The first step in L2F is to approximate the modified biexponential signal in the time domain, \eqref{eq:completesq}, in terms of the Hermite polynomials.
We have described this procedure in Section~\ref{sec:leastsq} (cf. \eqref{eq:coeffIntegral} and \eqref{eq:FindCoeff_SoE}). If $\nu$ is the Gauss quadrature measure, then $\mathbf{G}$ is the identity matrix, so that, in the notation of Section~\ref{sec:leastsq}, $d_k=\hat{f}(\nu;k)$. 
 We performed Gaussian-Hermite quadrature using the Chebfun software \cite{Driscoll2014} to calculate the required approximation coefficients $\hat{f}(\nu;k)$. 
In the preliminary numerical experiments presented below, we concentrated on the case $m=n=32$. Table \ref{CoeffExperiment} demonstrates the effect of using different values of $n$ in the presence of various SNR values. 

Having established an approximation for \eqref{eq:completesq}, we proceeded to find an approximation to its Fourier transform, \eqref{eq:fundarelation}.

As indicated in \eqref{eq:leastsqexpression} the approximation of this FT follows from \eqref{eq:hermitefour}, expressing the property of the Hermite functions as eigenfunction of the FT, with eigenvalues $(i)^{degree}$. We now filter this signal using a custom-designed finite-support filter \eqref{eq:measuresigma} \cite{singdet, trigwave,loctrigwave} to reduce sidelobes in the frequency domain, with the goal of identifying dominant peaks at the positions of the desired $\lambda$ values. 

\begin{figure}[H]
  \begin{minipage}[b]{0.5\textwidth}
    \includegraphics[width=\textwidth]{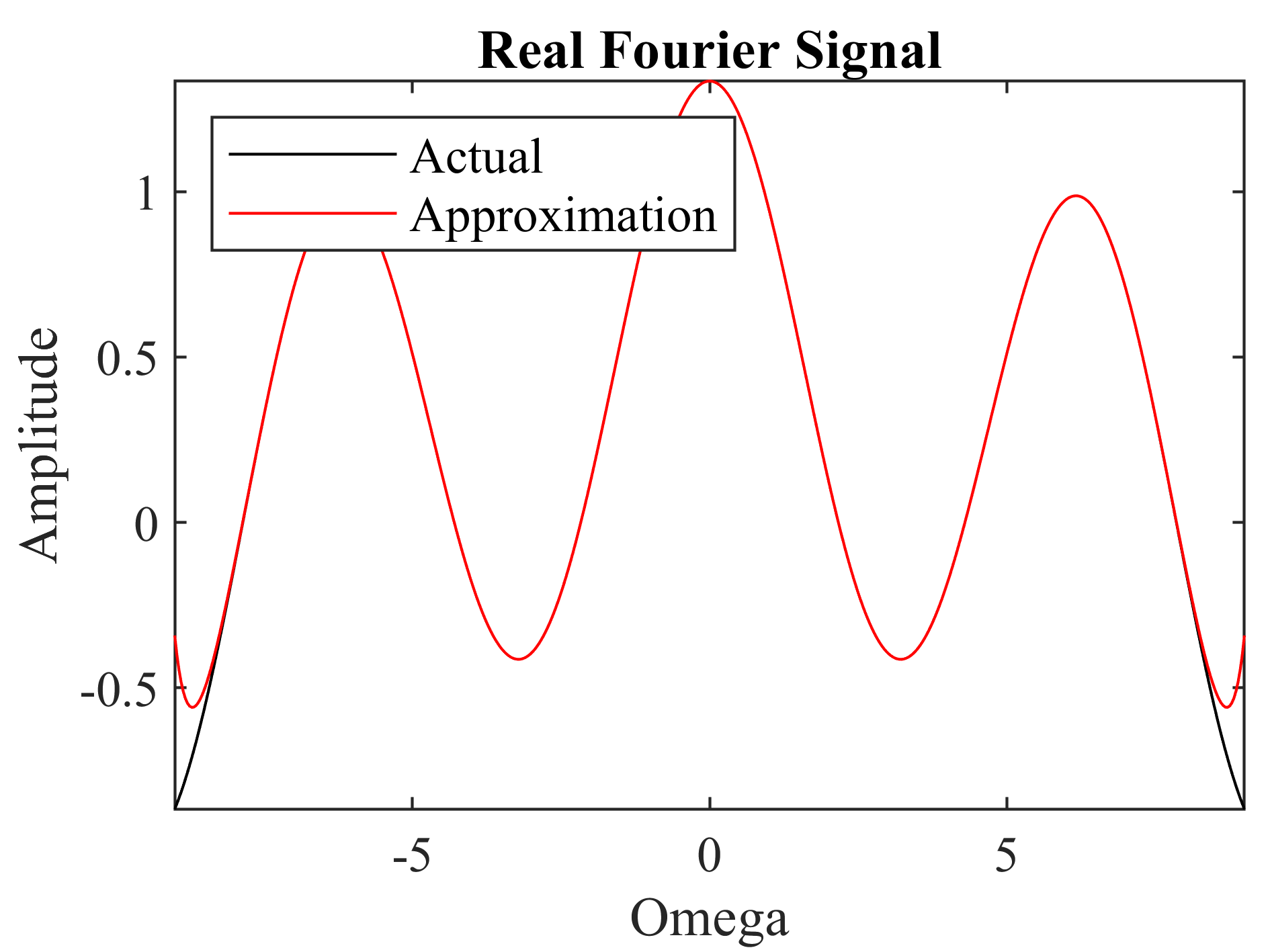}
  \end{minipage}
  \begin{minipage}[b]{0.5\textwidth}
    \includegraphics[width=\textwidth]{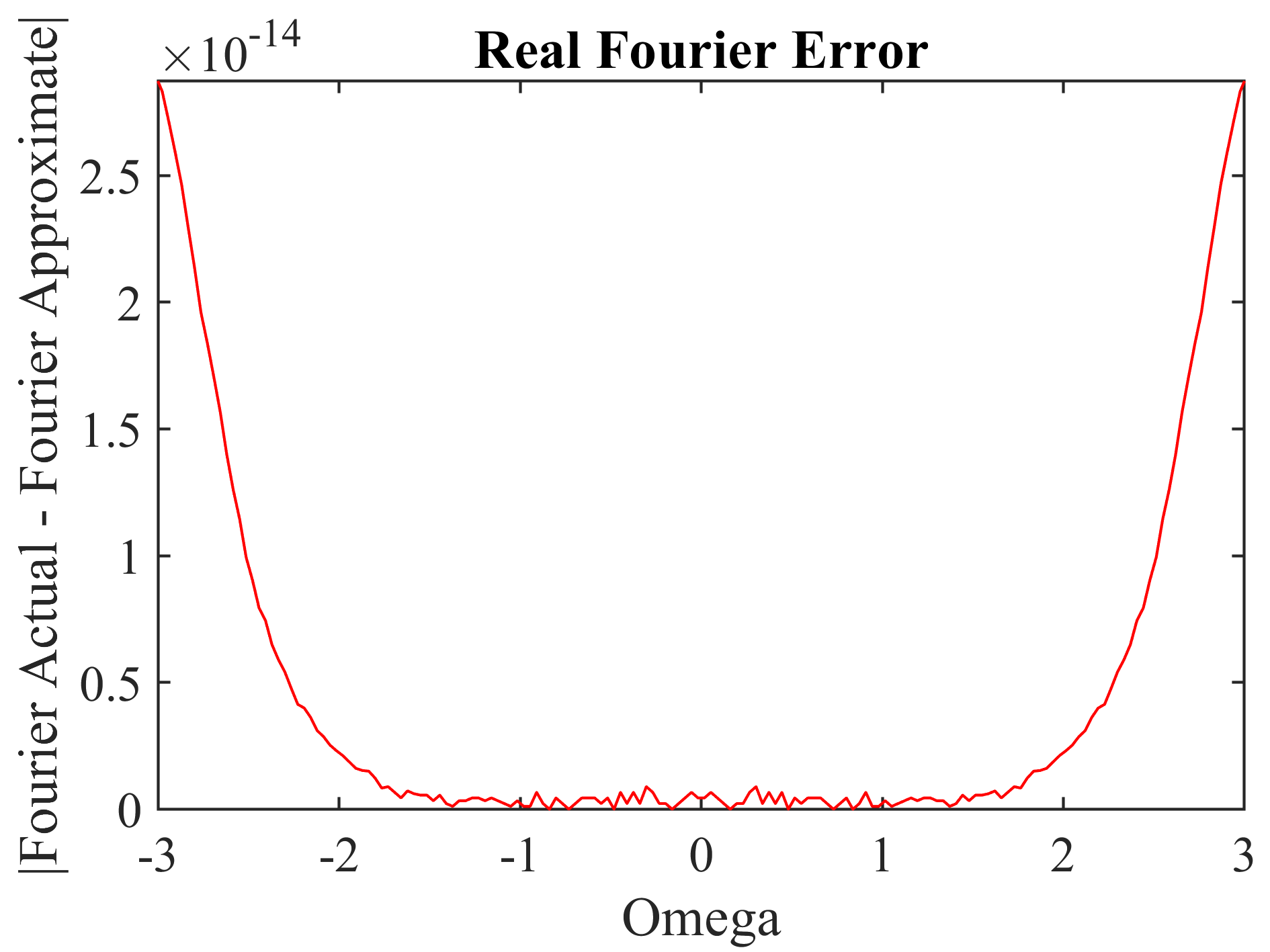}
  \end{minipage}
  \caption{The real part of the Fourier transform of the modified biexponential signal \eqref{eq:fundarelation}and its approximation through expansion with Hermite polynomials \eqref{eq:leastsqexpression}. The same signal parameters detailed in Figure \ref{fig:modgauss} are used. The fidelity of this approximation to the FT of the signal supports our attempt to extract a $\lambda$ value. However, the error in our approximation increases rapidly outside of a relatively small frequency window (right-hand panel).}
  \label{FourierFit}
\end{figure}
However, we found that for realistic values of $\lambda$ arising in MRI physics, the large range of $\omega$ values in \eqref{eq:fundarelation} required for resolution rendered this unworkable (Figs. \ref{fig:modgauss}, \ref{FourierFit}). However, as an alternative, the signal shift described above enabled us to identify the $\lambda$ corresponding to the component with the greater of the two $\lambda$'s; the other component was substantially more attenuated by the multiplicative exponential introduced in \eqref{eq:completesq}. Our approach then was to evaluate successive shifts to identify this dominant peak without the need for numerical optimization (e.g. nonlinear least squares analysis), thereby reducing the optimization problem from its original form with four unknowns, two of which enter nonlinearly, to three unknowns, one of which enters nonlinearly. 
  
Following estimation of the larger of the two time constants, which by convention we denote by $T_{22}$, we now complete our L2F method by estimating the remaining three parameters using NLLS. The reduction in dimensionality from four to three through the above procedure is expected to result in improved estimation of the remaining parameters. Thus, $T_{22}$ is inserted as a fixed parameter into the NLLS objective function and the remaining parameters, $A_1$,$A_2$, and $T_{21}$, are estimated using any one of the numerous available methods. We used the MATLAB lsqnonlin function, based on the trust region reflective algorithm with initial guesses for the $a_k$ taken as random values between 0 and 1 based on their interpretation as component fractions that sum to one. The initial guess for $T_{21}$ was set to a random value between 0ms and 300ms based on known properties of myelin-constrained water in MRI. Feasible set constraints were set between 0 and 1 for coefficients and 1 and 300ms for $T_{21}$. For the native four-parameter NLLS estimation problem, the same constraints and initial guesses used for $T_{21}$ were implemented for $T_{22}$; again, for L2F, $T_{22}$ was obtained from peak-finding without any other least-squares procedure.

The approximation of \eqref{eq:completesq} using an expansion with Hermite polynomials is shown in Figure \ref{fig:modgauss}. As seen, the Gauss-Hermite quadrature approach performs extremely well for approximation of the time-domain signal. This provides the basis for using of these coefficients to approximate the Fourier transform of equation \eqref{eq:completesq}.

\begin{figure}[H]
  \begin{minipage}[b]{0.5\textwidth}\includegraphics[width=\textwidth]{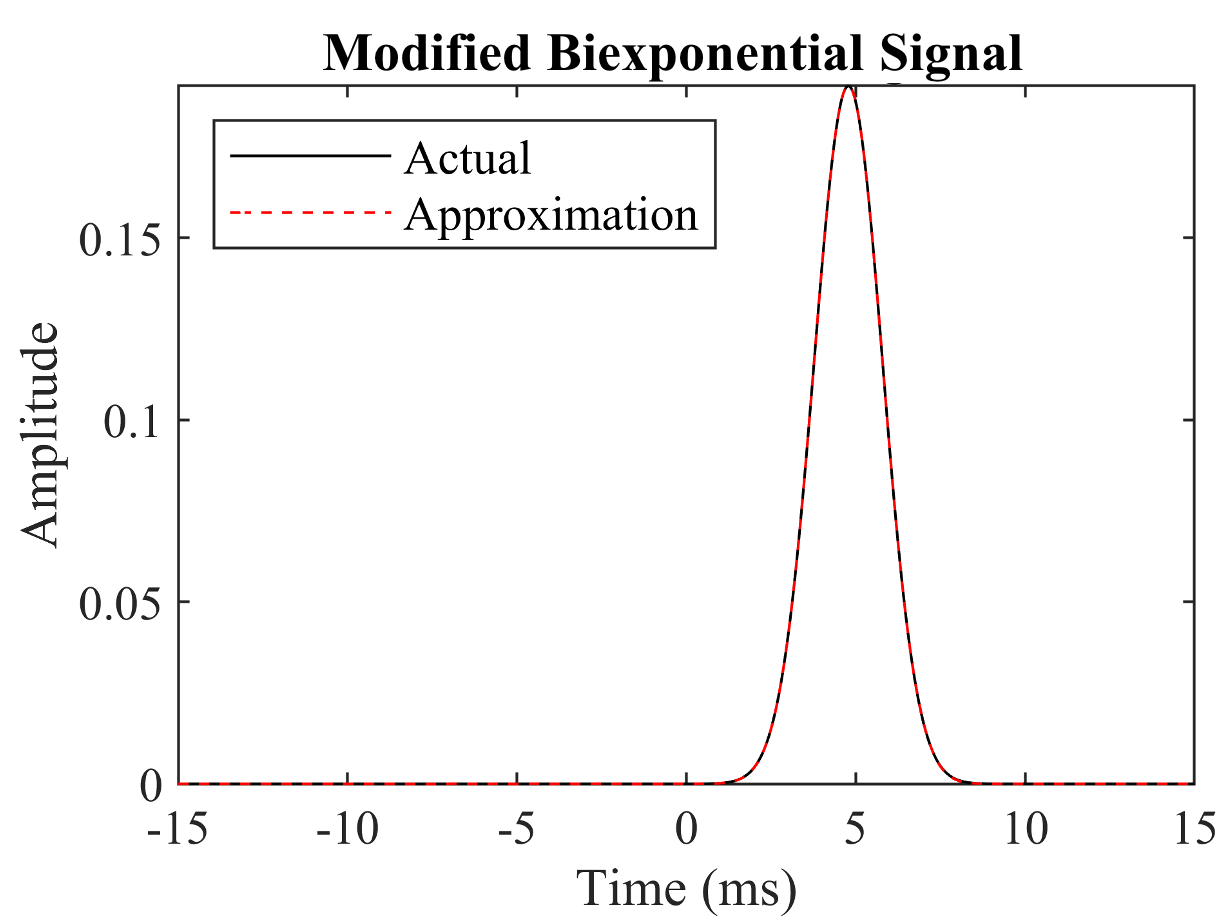} 
  \end{minipage}
  \begin{minipage}[b]{0.5\textwidth}
    \includegraphics[width=\textwidth]{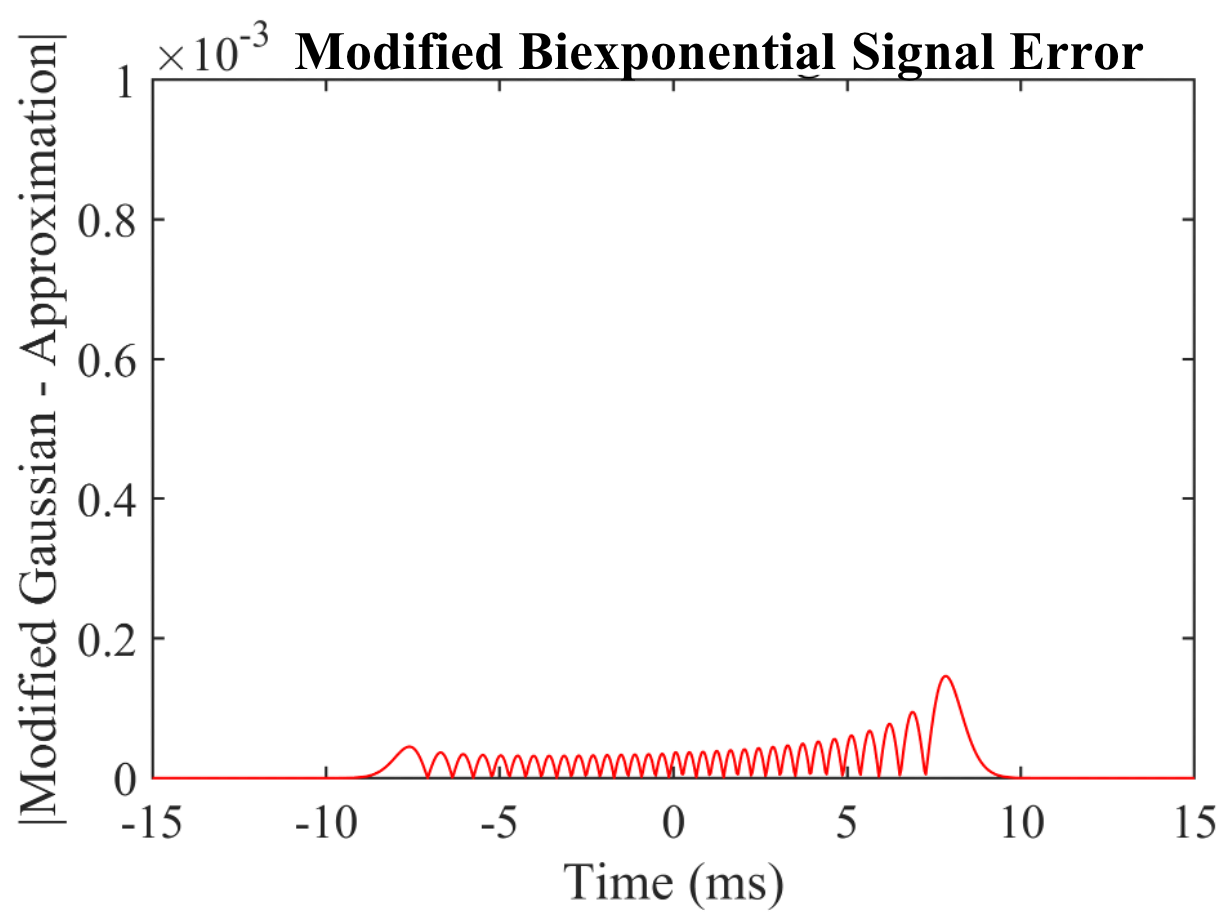}
  \end{minipage}
  \caption{Approximation of a modified biexponential \eqref{eq:completesq} decay curve with parameters: $A_1=0.5$, 
  $A_2=0.5$, $T_{21}=10$ms, $T_{22}=50$ms. The error for this approximation, shown in the right panel, remains small throughout the signal duration. }
  \label{fig:modgauss}
\end{figure}

The Fourier transform of the modified signal is presented in \eqref{eq:fundarelation}. The approximation of this signal, \eqref{eq:leastsqexpression}, is  nearly indistinguishable from the original in as seen in the difference plot in Fig. \ref{FourierFit}, right-hand panel. We find that the approximation maintains fidelity for frequencies in the domain of roughly [-6,6], although the displayed range in the difference plot (right-hand panel) is [-3,3] in order to preserve detail in the region about zero. Similar results were obtained for the imaginary part of the FT. 
The restricted range over which the approximation to the FT remains accurate presents severe constraints on our ability to resolve closely spaced signals. As an alternative, we make use of the signal shift as described in \eqref{eq:orignalsignal} to attenuate the amplitude of the peak corresponding to the smaller value of $\lambda$, which is $\lambda_2$ according to our convention in which $T_{22}>T_{21}$, so that the value of $\lambda_1$ can be readily identified as the dominant remaining peak in the FT. A visualization of this procedure, which is detailed in Section \ref{sec:trigkernel}, is shown in Figure \ref{fig:lambdaextraction}.

\begin{figure}[H]
    \centering \includegraphics[width=0.6\textwidth]{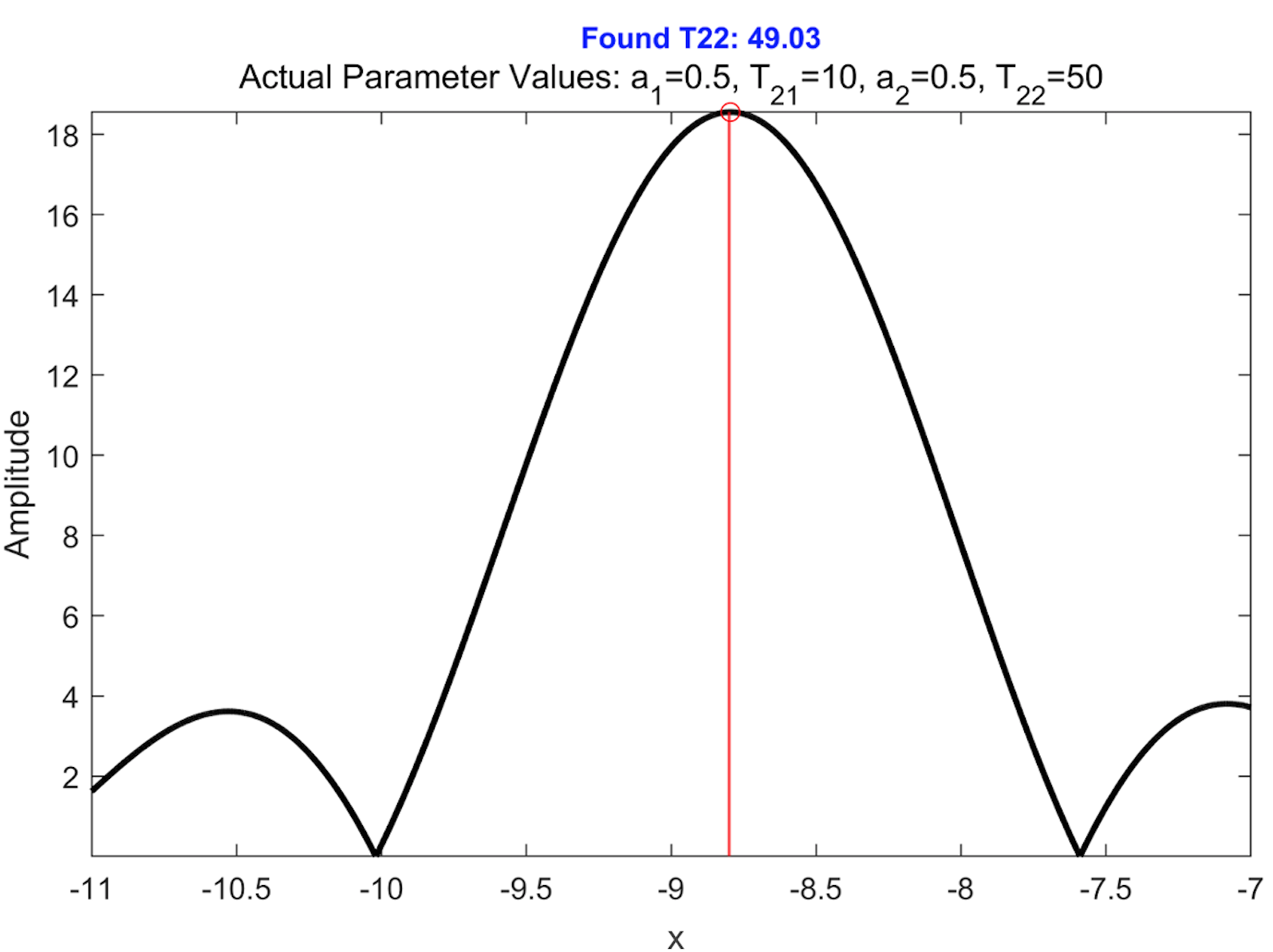}
    \caption{Illustration of $T_{22}$ estimation.The same signal parameters detailed in Figure \ref{fig:modgauss} are used. The vertical red line shows the theoretical value of $\lambda_1$ in terms of $x$ (see Section A) while the red circle shows the location of the point of maximum value (the estimated value of $\lambda$). For these parameters, a $T_{22}$ value of approximately 49.03 was found, close to the correct underlying value of 50.}
    \centering
    \label{fig:lambdaextraction}
\end{figure}

Further numerical results were generated for time decay constants which are more closely spaced, and hence more problematic to extract (see Table \ref{table4060} and Table \ref{table5060}). To model experimental data, we performed our analysis with different levels of Gaussian noise added to the original biexponential signal. Thus, results in these tables were obtained by adding noise to a biexponential signal, with SNR defined as the signal peak (initial) value divided by the standard deviation of the added Gaussian noise. Parameters were then extracted from this signal using the NLLS and L2F procedures described above. This process was performed for noiseless signals as well as over 500 noise realizations for SNR values of $10^6$, $10^5$, and $10^4$. The mean, standard deviation (StDev), and root mean squared error (RMSE) with respect to the correct underlying value for each estimated parameter across noise realizations are shown in the Tables.  
\vspace{-8pt}
\begin{table}[H]
    \includegraphics[width=\textwidth]{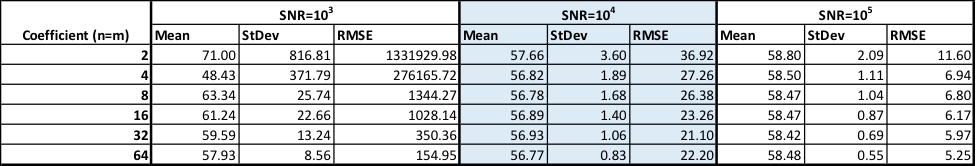}
    \caption{Estimation of $T_{22}$ for varying levels of noise, with underlying values of $T_{21}=40$ms, $T_{22}=60$ms, and $A_1$=$A_2$=0.5. The mean, standard deviation (StDev), and root mean squared error (RMSE) for $T_{22}$ are calculated from 500 independent noise realizations. Results are shown for three values of SNR. Results improved overall with increasing values of n and m up to $m=n=32$ or $64$. A value of 32 was selected for these preliminary numerical demonstrations.}
    \label{CoeffExperiment}
\end{table}
Table 1 shows results for decay constants of $T_{21}=40$ms and $T_{22}=60$ms, with  $A_1 = A_2=0.5$. For the noiseless signal, there is no variability in the prediction of $T_{22}$ using the L2F method. Variability of L2F's remaining parameters-as well as all of the NLLS predictions-is attributable to the randomness of initial guesses and the presence of local minima in the optimization procedure used to obtain those estimates. In the noiseless case and the nearly-noiseless case of SNR=$10^6$, our results are superior to those obtained from NLLS.  However, with increasing noise, error in the L2F method increases rapidly due to the extent of shifting required to extract a single decay constant; recall that, due to of the rapidly decaying exponential in \eqref{eq:completesq}, this shift attenuates the amplitude of the peak corresponding to the larger of the two $\lambda$'s (i.e. the smaller of the two $T_{2}$'s), permitting the identification of the smaller $\lambda$ (i.e. the larger $T_2$). Given that shifting also attenuates the peak related to the smaller $\lambda$ (though not to the same extent as the larger $\lambda$), we see that introducing too much of a shift causes noise in the signal to become dominant.

\begin{table}[H]
    \includegraphics[width=\textwidth]{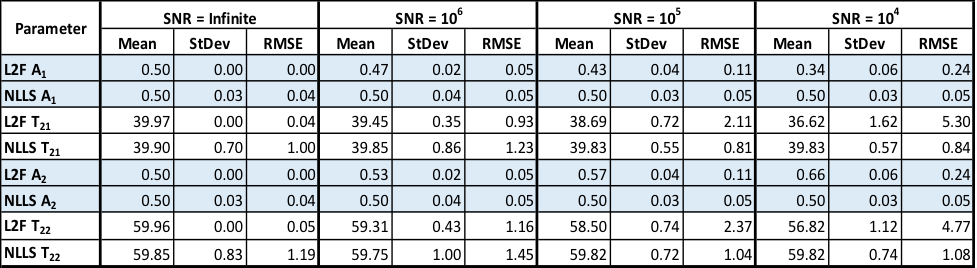}
    \caption{Parameter estimation for varying levels of noise, where $T_{21}=40$ms, $T_{22}=60$ms, and $A_1$=$A_2$=0.5. The mean, standard deviation (StDev), and root mean squared error (RMSE) for each parameter are calculated from 500 independent noise realizations. Results are shown for four values of SNR. In the leftmost column, we see that for a noisless signal the L2F method has smaller StDev and RMSE results compared to NLLS for all four parameters-indicating the superiority of this method for this ideal case. Performance of L2F is seen to decrease even for the very high SNR value of SNR=$10^6$.For SNR $\le$  $10^5$, L2F is seen to perform worse than NLLS for all parameters, as shown by the corresponding values of StDev and RMSE.}
    \label{table4060}
\end{table}

\begin{table}[H]
    \includegraphics[width=\textwidth]{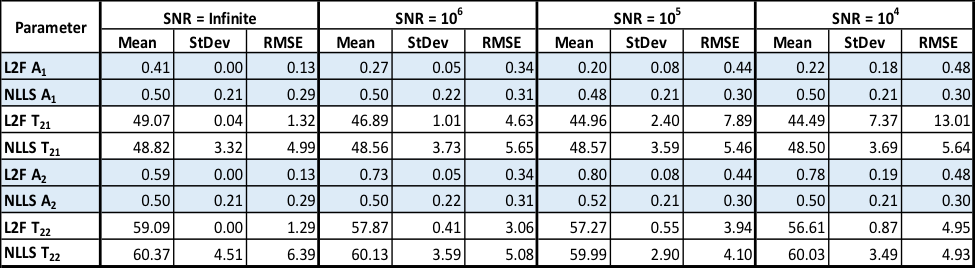}
    \caption{Parameter estimation for varying levels of noise, where $T_{21}=50$ms, $T_{22}=60$ms, and $A_1$=$A_2$=0.5. Abbreviations are the same as those in Table \ref{table4060}. In the case of a noisless signal, for all parameters, L2F outperforms NLLS in terms of StDev and RMSE values. As noise increases, L2F performance worsens, due to the increased StDev of the initial $T_{22}$ estimate that is the basis for the method.}
    \label{table5060}
\end{table}

It is notable that NLLS performs essentially equivalently across all of these levels of noise, all of which represent signals with very high SNR. However, while the noiseless behavior of L2F, and its behavior for SNR = $10^6$, were superior to NLLS, L2F exhibited more rapid degradation of overall performance as SNR decreased. We attribute this decline in performance to the increasing variability in the estimation of $T_{22}$. We see that in the noiseless case, there is no variability in the $T_{22}$ prediction. However, upon the introduction of even a small amount of noise, L2F's estimation of $T_{22}$ exhibits considerable variability, as seen from the given StDev values for the estimates. Recall that the remaining parameters are estimated using the fixed value of $T_{22}$ recovered from L2F, so that variability in the latter leads variability in these remaining parameters:$A_1$, $A_2$, and $T_{21}$. Viewing \ref{table4060}, as SNR decreases, standard deviation and RMSE values are fairly constant for the NLLS method when comparing the results of parameter estimation at high and low SNRs. Performing the same comparison with the L2F method, we observe that while L2F outperforms NLLS for the highest values of SNR, its peformance degrades and is not as robust as NLLS for noisier signals.

Table \ref{table5060} presents the same analysis for a case in which the decay constants of the two signal components are yet more closely spaced:$T_{21}=50$ms and $T_{22}=60$ms. Again, the L2F approach outperforms NLLS for noiseless signals and for near-infinite SNR due to its stable prediction of the $T_{22}$ parameter. This indicated the theoretical efficacy of L2F even in the regime where NLLS is unable to effectively distinguish between closely-spaced decay constants. As noise increases, the increasing StDev of L2F's $T_{22}$ prediction reduces the performance of this method. $T_{21}$ appears to suffer the most and the coefficients $A_1$ and $A_2$ are also affected. Continuing our analysis of the SNR=$10^4$ case, we see that while the StDevs of these coefficients are smaller than those predicted by NLLS, bias in the values of these predictions is the driver of their poor performance with the L2F method. 

Finally, we recognize that the multiexponential model we are working with in this paper represents a potential oversimplification of the signal model in an actual MRI experiment.  In particular, the quantification of the estimated parameters, including in  the related method of T2 relaxation spectrum analysis mentioned above in the Introduction, depends on several complexities in the signal generation process and has been explored in previous work\cite{Akhondi,Kirtil2016,Kumar2016, Mehdizadeh2022,Wiggerman2020} . However, while such experiments, and in particular applications to neuroscience, form the basis for our motivation in this study, we are here  concerned with the classic, and widely-applicable, multiexponential signal as written in \eqref{eq:observations} and as interpreted throughout our manuscript.  

\section{Proofs}\label{sec:proofs}

In order to prove Theorem~\ref{theo:fourdegapprox}, we need some preparation. 
 First, we recall some facts from the theory of weighted polynomial approximation.
 Parts (a) and (b) are given (in a greater generality) in \cite[Theorem~3.2.5]{mhasbk}. Part (c) follows from Proposition~\ref{prop:mrs} and \eqref{eq:hermitechristbd}.  Part (d) is easy to deduce using \cite[Thoerem~3.4.2(b), Lemma~3.4.4(b)]{mhasbk}. 
 
\begin{lemma}\label{lemma:hermitelemma}
Let $n\ge 1$ be an integer.\\
\noindent
{\rm (a)} We have
\be\label{eq:hermitechristbd}
K_n(x,x)=\sum_{j=0}^{n-1} \psi_j(x)^2   \ls n^{1/2}, \qquad x\in\RR.
\ee
{\rm (b)}
There exists $C>0$ such that for $|x|\le Cn^{1/2}(1+c/n)$,
\be\label{eq:christlowbd}
K_n(x,x)\gs \sqrt{n}.
\ee
{\rm (c)} If  $|y|>2n^{1/2}$, then
\be\label{eq:mrshatkn}
|\hat{K}_n(w,y)|\ls\exp(-cn), \qquad w\in\RR.
\ee
{\rm (d)} For any $P\in \Pi_n$, $1\le p\le\infty$, we have the Bernstein inequality:
\be\label{eq:bernstein}
\|P'\|_p \ls \sqrt{n}\|P\|_p.
\ee
\end{lemma}

\begin{lemma}\label{lemma:ballmeasure}
Let $m\ge 1$ be an integer, and $\nu\in SZM(m)$. With $C$ as in Lemma~\ref{lemma:hermitelemma}(b), we have for $|x|\le C\sqrt{m}$, and for some constant $C_1>0$,
\be\label{eq:ballmeasure}
\nu([x-C_1/\sqrt{m},x+C_1/\sqrt{m}]) \ls m^{-1/2}.
\ee
\end{lemma}

\begin{proof}\ 
Let $|x|\le C\sqrt{m}$.  
Using the Bernstein inequality \eqref{eq:bernstein}, the Schwarz inequality,  \eqref{eq:hermitechristbd}, and \eqref{eq:christlowbd}, we see that
$$
\begin{aligned}
|K_m(x,y)-K_m(x,x)|&\le |x-y|\|K_m'(x,\circ)\|_\infty \ls \sqrt{m}|x-y|\|K_m(x,\circ)\|_\infty\\
& \ls \sqrt{m}|x-y|\left\{\sum_{k=0}^{m-1}\psi_k(x)^2\right\}^{1/2}\left\{\max_{y\in\RR}\sum_{k=0}^{m-1}\psi_k(y)^2\right\}^{1/2}\\
&\ls m|x-y|K_m(x,x)^{1/2} \ls \sqrt{m}K_m(x,x)|x-y|.
\end{aligned}
$$
Consequently, for a judiciously chosen $C_1>0$,  we have for $|x-y|\le C_1/\sqrt{m}$: 
\be\label{eq:pf3eqn1}
\sqrt{m}\ls (1/2)K_m(x,x)\le K_m(x,y)\le (3/2)K_m(x,x)\ls \sqrt{m}.
\ee
Hence, using \eqref{eq:strongsmz} and \eqref{eq:hermitechristbd} again,
\be\label{eq:pf3eqn2}
\begin{aligned}
m\nu([x-C_1/\sqrt{m},x+C_1/\sqrt{m}])&\le \int_{x-C_1/\sqrt{m}}^{x+C_1/\sqrt{m}}K_m(x,y)^2d\nu(y)\le \int_\RR K_m(x,y)^2d\nu(y)\\
&\ls \int_\RR K_m(x,y)^2dy=K_m(x,x)\ls \sqrt{m}.
\end{aligned}
\ee
This proves  \eqref{eq:ballmeasure}.
\end{proof}

The following lemma establishes bounds on the integral norms of the kernels defined in \eqref{eq:darboux} and \eqref{eq:fourdarboux}.
\begin{lemma}\label{lemma:kernelbds}
Let $m\ge n\ge 1$ be  integers, such that $C^2m\ge 4n$, where $C$ as in Lemma~\ref{lemma:hermitelemma}(b),  and $\nu\in SMZ(n)$. Then for every $x\in\RR$,
\be\label{eq:christbd}
\sum_{j=0}^{n-1} |\tilde{p}_j(x)|^2\sim \sum_{j=0}^{n-1} \psi_j(x)^2   \ls n^{1/2}, \quad 
\ee
\be\label{eq:fourchristbd}
\int |\hat{K}_n(\nu;x,y)|^2d\nu(y)=\sum_{j=0}^{n-1} |\mathfrak{F}(\tilde{p}_j)(x)|^2 \sim \sum_{j=0}^{n-1} \psi_j(x)^2   \ls n^{1/2}.
\ee
Moreover,
\be\label{eq:kernelbds}
\int |K_n(\nu;x,y)|d\nu(y)  \ls n^{1/2}, 
\ee
\be\label{eq:kernbd}
\int |K_n(x,y)|d\nu(y)\ls n^{1/2}.
\ee
Since the Lebesgue measure satisfies \eqref{eq:strongsmz} for every $n$, we have
\be\label{eq:lebesgueconst}
\max\left(\int |K_n(x,y)|dy, \int |K_n(\nu;x,y)|dy\right) \ls n^{1/2}, \qquad n=1,2,\cdots.
\ee
Moreover
\be\label{eq:fourkernbds}
\max\left(\int |\hat{K}_n(\nu;x,y)|dy, \int |\hat{K}_n(\nu;x,y)|d\nu(y)\right) \ls n^{1/2}
\ee
\end{lemma}

\begin{proof}
 In view of \eqref{eq:darboux},
\be\label{eq:pf1eqn1}
\sum_{j=0}^{n-1} |\tilde{p}_j(x)|^2=\sum_{\ell,\ell'=0}^{n-1}\mathbf{G}^{-1}_{\ell,\ell'}\psi_\ell(x)\psi_{\ell'}(x),
\ee
which is a quadratic form for $\mathbf{G}^{-1}$.
So, the estimate \eqref{eq:christbd} follows by using \eqref{eq:invGnorm} with $\psi_\ell(x)$ in place of $d_\ell$.
The second estimate in \eqref{eq:christbd} is the same as \eqref{eq:hermitechristbd}.
The estimates in \eqref{eq:fourchristbd} are obtained similarly, noting that
$$
\hat{K}(\nu;x,y)=\sum_{\ell,\ell'=0}^{n-1}\mathbf{G}^{-1}_{\ell,\ell'}(-i)^\ell \psi_\ell(x)\psi_{\ell'}(x)
$$
To prove \eqref{eq:kernelbds}, we note that since $2\sqrt{n}\le C\sqrt{m}$, Lemma~\ref{lemma:ballmeasure} shows that 
$$
 |\nu|([-2\sqrt{n},2\sqrt{n}])\ls \sqrt{nm}\sqrt{1/m}\ls \sqrt{n}.
$$
Hence,   using Proposition~\ref{prop:mrs}, with $K_n(\nu;x,y)$ in place of $P$ and $p=1$, followed by  Schwarz inequality and \eqref{eq:christbd},we deduce that
$$
\begin{aligned}
\int_\RR |K_n(\nu;x,y)|d|\nu|(y)&\ls \int_{-2\sqrt{n}}^{2\sqrt{n}} |K_n(\nu;x,y)|d|\nu|(y) \\
&\ls \left\{\int_{-2\sqrt{n}}^{2\sqrt{n}} |K_n(\nu;x,y)|^2d|\nu|(y)\right\}^{1/2}\left\{ |\nu|([-2\sqrt{n},2\sqrt{n}])\right\}^{1/2}\ls n^{1/2}.
\end{aligned}
$$
The estimates \eqref{eq:kernbd} (and hence, \eqref{eq:lebesgueconst}) follow similary. 
 The estimate \eqref{eq:fourkernbds} follows similarly as well, using  \eqref{eq:fourchristbd}) and \eqref{eq:nuorthorel}.
\end{proof}

We note that we are working with a least squares approximation, but require estimates in the uniform norm, while the Fourier transform of a function is bounded uniformly by the $L^1$ norm of the function.
To facilitate the transitions among these norms, we recall the following lemma \cite[Theorem~4.2.4]{mhasbk} relating the various norms for weighted polynomials.
\begin{lemma}\label{lemma:nikolskii}
For any integer $n\ge 1$ and $P\in \Pi_n$, we have
\be\label{eq:nikolskii}
\|P\|_1\ls n^{1/2}\|P\|_\infty\ls n\|P\|_1.
\ee
\end{lemma}

Next, we obtain some constructions which approximate a function $f$ in different norms at the same time; i.e., constructions which do not depend upon the norm in which the degree of approximation is measured.
 It is not difficult to construct a near-best approximation to $f$ from $\Pi_n$.
 Let $n'=\lfloor n/2\rfloor$.
For $f\in L^1+L^\infty$, we write
$$
\hat{f}(k)=\int f(y)\psi_k(y)dy,
$$
$$
s_\ell(f)=\sum_{k=0}^{\ell-1}\hat{f}(k)\psi_k,
$$
and 
\be\label{eq:valleepoussin}
V_n(f)=\frac{1}{n-n'}\sum_{\ell=n'}^{n-1} s_\ell(f).
\ee
Part (a) of the following lemma is proved in \cite[Lemma~4.1.5]{mhasbk}. In the present manuscript, $\exp(-x^2/2)$, $f$, and  $V_n(f)$ take the place of $w$, $wf$ and $wv_n(f)$ in that reference. Part (b) follows easily from \eqref{eq:lebesgueconst} and the Riesz-Thorin interpolation theorem.
\begin{lemma}\label{lemma:goodapprox}
If $1\le p\le \infty$, $f\in L^p$, then\\
{\rm (a)}
\be\label{eq:goodapprox}
E_{p,n}\le \|f-V_n(f)\|_p \ls E_{p,n/2}(f).
\ee
and\\
{\rm (b)}
\be\label{eq:fourprojerr}
E_{p,n}\le \|f-s_n(f)\|_p \ls n^{1/2}E_{p,n}(f).
\ee
\end{lemma}

With this preparation, we are now ready to prove Theorem~\ref{theo:fourdegapprox}.\\

\noindent\textsc{Proof of Theorem~\ref{theo:fourdegapprox}.} 
In view of \eqref{eq:leastsqexpression} and Lemma~\ref{lemma:kernelbds}, we see that for any $x\in\RR$,
$$
|S_n(\nu;f)(x)|=\left|\int_\RR f(y)K_n(\nu;x,y)d\nu(y)\right|\le \|f\|_\infty\int_\RR |K_n(\nu;x,y)|d|\nu|(y)\ls n^{1/2}\|f\|_\infty;
$$
i.e.,
$$
\|S_n(\nu;f)\|_\infty \ls n^{1/2}\|f\|_\infty.
$$
Since $S_n(\nu;f)\in\Pi_n$, 
Lemma~\ref{lemma:nikolskii}  leads to
\be\label{eq:pf2eqn1}
\|S_n(\nu;f)\|_1\ls n^{1/2}\|S_n(\nu;f)\|_\infty \ls n\|f\|_\infty.
\ee
Next, we observe that for any $P\in \Pi_n$,
\be\label{eq:reprod}
S_n(\nu;P)(x)=\int P(y)K_n(\nu;x,y)d\nu(y)=P(x).
\ee
Hence, using \eqref{eq:pf2eqn1}, we have for \emph{any} $P\in\Pi_n$,
\be\label{eq:pf2eqn2}
\|f-S_n(\nu;f)\|_1=\|f-P-S_n(\nu;f-P)\|_1\le \|f-P\|_1 +\|S_n(\nu;f-P)\|_1 \ls \|f-P\|_1 +n\|f-P\|_\infty.
\ee
We use this estimate with $V_n(f)$ in place of $P$ and use Lemma~\ref{lemma:goodapprox}(a) to arrive at \eqref{eq:fourdegapprox}.
If  we use $s_n(f)$ in place of $P$ in \eqref{eq:pf2eqn2} instead, then Lemma~\ref{lemma:goodapprox}(b) leads to \eqref{eq:fourprojapprox}.
\qed

Next, we turn to the proof of Theorem~\ref{theo:noiselessfour}, which again requires some further preparation.

The generating function for Hermite polynomials is given by (cf. \cite[Formula~(5.5.7)]{szego}))
\be\label{eq:generating_function}
\exp(2xw-w^2)=\pi^{1/4}\sum_{k=0}^\infty \frac{2^{k/2}}{\sqrt{k!}}h_k(x)w^k.
\ee
If $\lambda>0$, then using $w=-\lambda/2$ in \eqref{eq:generating_function}, we obtain
$$
\exp(-x\lambda)=\pi^{1/4}\exp(\lambda^2/4)\sum_{k=0}^\infty \frac{(-1)^k}{2^{k/2}\sqrt{k!}}h_k(x)\lambda^k,
$$
so that
\be\label{eq:modgenfn}
g_\lambda(x)=\exp(-x\lambda-x^2/2)=\pi^{1/4}\exp(\lambda^2/4)\sum_{k=0}^\infty \frac{(-1)^k}{2^{k/2}\sqrt{k!}}\psi_k(x)\lambda^k.
\ee
Next, we recall  that
\be\label{eq:hermitenorms}
\|\psi_k\|_1\ls k^{1/4}, \qquad \|\psi_k\|_\infty \ls k^{-1/4}.
\ee
We observe further using the Schwarz inequality that for any $x\in\RR$,
\be\label{eq:mockexp}
\sum_{k=0}^\infty\frac{x^k}{\sqrt{k!}}=\sum_{k=0}^\infty\frac{(\sqrt{2}x)^k}{\sqrt{k!}}2^{-k/2}\le \left\{\sum_{k=0}^\infty \frac{(\sqrt{2}x)^{2k}}{k!}\right\}^{1/2}\left\{\sum_{k=0}^\infty 2^{-k}\right\}^{1/2}=\exp(x^2).
\ee
Finally, we note that 
$$
\frac{(n+k-1)!}{(n-1)!}=k!\left(\gattha{n+k-1}{n-1}\right)\ge k!.
$$
Using \eqref{eq:hermitenorms}, \eqref{eq:mockexp}, and the above estimate, we deduce that
\be\label{eq:l1errest}
\begin{aligned}
\exp(\lambda^2/4)E_{1,n}(g_\lambda) &\ls \sum_{k=n}^\infty \frac{k^{1/4}}{2^{k/2}\sqrt{k!}}\lambda^k
\le \sum_{k=n}^\infty \frac{1}{k^{1/4}\sqrt{(k-1)!}}(\lambda/\sqrt{2})^k
\le n^{-1/4}\sum_{k=n}^\infty \frac{1}{\sqrt{(k-1)!}}(\lambda/\sqrt{2})^k\\
&\le \frac{(\lambda/\sqrt{2})^n}{n^{1/4}\sqrt{(n-1)!}}\sum_{k=0}^\infty \frac{1}{\sqrt{k!}}(\lambda/\sqrt{2})^k
\ls \frac{(\lambda/\sqrt{2})^n}{n^{1/4}\sqrt{(n-1)!}}\exp(\lambda^2/2).
\end{aligned}
\ee
We estimate $E_{\infty,n}(g_\lambda)$ similarly and conclude that
\be\label{eq:errest}
E_{1,n}(g_\lambda) \ls \frac{n^{1/4}}{\sqrt{n!}}(\lambda/\sqrt{2})^n\exp(\lambda^2/4), \qquad E_{\infty, n}(g_\lambda)\ls \frac{n^{-1/4}}{\sqrt{n!}}(\lambda/\sqrt{2})^n\exp(\lambda^2/4).
\ee
With this preparation, Theorem~\ref{theo:noiselessfour} is now easy to prove.\\

\textsc{Proof of Theorem~\ref{theo:noiselessfour}.}
We  apply Theorem~\ref{theo:fourdegapprox} with $f$ defined in \eqref{eq:completesq}, taking $\nu$ to be supported on $[-L,R]$ (i.e., $\mathsf{supp}(\nu)\subseteq [-L,R]$).
We note that in view of Remark~\ref{rem:relevance}, we may assume that $f$ is defined on the entire real axis rather than just on $[-L,R]$.
We note that
$$
f(t)=\sum_{k=1}^K a_k g_{\lambda_k}(t).
$$
Hence, Theorem~\ref{theo:noiselessfour} is a simple consequence of Theorem~\ref{theo:fourdegapprox} and the estimates \eqref{eq:errest}. \qed

The next objective of this section is to estimate $\int E(y)\hat{K}_n(w,y)d\nu(y)$.
We will use the following lemma, H\"offding's inequality (\cite[Theorem~2.8]{boucheron2013concentration}), several times. Note that this inequality quantifies the degree to which the sum of the indicated random variables can exceed a given value; it is an example of a concentration inequality.    
\begin{lemma}\label{lemma:hoeffding}
(\textbf{H\"offding's inequality})  Let $X_1,\cdots, X_M$ be i.i.d. random variables with mean $0$, and each $X_k\in [a_k,b_k]$. Then for any $t>0$,
\be\label{eq:hoeff}
\mathsf{Prob}\left\{\left|\sum_{k=1}^M X_k\right|\ge t\right\}\le 2\exp\left(-\frac{t^2}{\sum_{k=1}^M (b_k-a_k)^2}\right).
\ee
\end{lemma}
We also need a covering lemma (see, e.g., \cite{cloninger2020cautious}).
\begin{lemma}\label{lemma:covering}
There exists a set $\C$ with at most $\sim n$ elements such that for any $P\in\Pi_n$,
\be\label{eq:coverning}
\|P\|_\infty \sim \max_{y\in \C}|P(y)|.
\ee
\end{lemma}
With these preliminaries, we can now prove Theorem~\ref{theo:noisetheo}.\\

\noindent
\textsc{Proof of Theorem~\ref{theo:noisetheo}.}

We observe first that in view of \eqref{eq:mztotalvariation} and \eqref{eq:mrs},
\be\label{eq:pf4eqn1}
\left\|\int_{|y|\ge C\sqrt{m}}E(y)\hat{K}_n(\circ,y)d\nu(y)\right\|_\infty\ls \left\|\int_{|y|\ge C\sqrt{m}}|E(y)||\hat{K}_n(\circ,y)|d\nu(y)\right\|_\infty\ls Sm^{1/2}\exp(-cn).
\ee
Now, let $w\in\RR$ be fixed for the time being (until \eqref{eq:pf4eqn6} is proved).
 With $C_1$ as in Lemma~\ref{lemma:ballmeasure}, we divide the interval $[-C\sqrt{m}, C\sqrt{m}]$ into subintervals $I_k$, each of length at most $2C_1/\sqrt{m}$, so that $|\nu|(I_k) \ls m^{-1/2}$. 
 We denote the characteristic (indicator) function of $I_k$ by $\chi_k$. We then consider the random variables
\be\label{eq:pf4eqn2}
X_k=\int_{I_k} E(y)\hat{K}_n(w,y)d\nu(y)=\int E(y)\chi_k(y)\hat{K}_n(w,y)d\nu(y),
\ee
so that
\be\label{eq:pf4eqn3}
\int_{|y|\le Cm^{1/2}}E(y)\hat{K}_n(w,y)d\nu(y)=\sum_k X_k.
\ee
Since the expected value of each $E(y)$ is $0$, that of each $X_k$ is $0$ as well.
Further, for each $k$,
\be\label{eq:pf4eqn4}
|X_k|=\left|\int E(y)\chi_k(y)\hat{K}_n(w,y)d\nu(y)\right|\le S\int |\chi_k(y)\hat{K}_n(w,y)|d\nu(y).
\ee

Using the Schwarz inequality and Lemma~\ref{lemma:ballmeasure}, we see that
\be\label{eq:pf4eqn5}
\left(\int |\chi_k(y)\hat{K}_n(w,y)|d\nu(y)\right)^2\le \nu(I_k)\left\{\int_{I_k} |\hat{K}_n(w,y)|^2d\nu(y)\right\}\ls m^{-1/2}\left\{\int_{I_k} |\hat{K}_n(w,y)|^2d\nu(y)\right\};
\ee
i.e., with 
$$
r_k= S m^{-1/4}\left\{\int_{I_k} |\hat{K}_n(w,y)|^2d\nu(y)\right\}^{1/2},
$$
$$
|X_k|\ls r_k \qquad\mbox{for all } k.
$$
Using   \eqref{eq:fourchristbd}, we conclude that
\be\label{eq:pf4eqn8}
\sum_k r_k^2\ls S^2(n/m)^{1/2}.
\ee
So, H\"offding's inequality implies that
\be\label{eq:pf4eqn6}
\mathsf{Prob}\left\{\left|\int_{|y|\le Cm^{1/2}}E(y)\hat{K}_n(w,y)d\nu(y)\right|\ge t\right\}=\mathsf{Prob}\left\{\left|\sum_k X_k\right|\ge t\right\}\le 2\exp\left(-\frac{t^2}{S^2(n/m)^{1/2}}\right).
\ee
We note that this estimate holds for each $w\in\RR$, and $\sum_k X_k\in \Pi_n$. 
In view of the covering lemma, Lemma~\ref{lemma:covering}, this leads to
\be\label{eq:pf4eqn7}
\mathsf{Prob}\left\{\left\|\int_{|y|\le Cm^{1/2}}E(y)\hat{K}_n(\circ,y)d\nu(y)\right\|_\infty\ge t\right\}\ls n\exp\left(-\frac{ct^2}{S^2(n/m)^{1/2}}\right).
\ee
Choosing 
$$
t=cS(n/m)^{1/4}\sqrt{\log (n/\delta)},
$$
we deduce that with probability $>1-\delta$,
$$
\left\|\int_{|y|\le Cm^{1/2}}E(y)\hat{K}_n(\circ,y)d\nu(y)\right\|_\infty\ls S(n/m)^{1/4}\sqrt{\log (n/\delta)}.
$$
Together with \eqref{eq:pf4eqn1}, we obtain \eqref{eq:singleEest} with a judicious choice of $n$.
\qed

\section{Conclusions}
The analysis of multiexponential decaying signals is a longstanding problem in mathematics and physics. We have introduced a novel method for parameter estimation in this setting based on the fact that the Hermite polynomials are eigenfunctions of the Fourier transform. This enables us to avoid direct use of the notoriously ill-posed inverse Laplace transform for multiexponential analysis, although as indicated by our numerical results, the transformation to the Fourier domain retains the native ill-posedness of the original problem. Our numerical results, presented for the special case of biexponential decay, demonstrate the success of the method for a wide range of decay constants, including in regimes where the more traditional method of NLLS fails.  Again, the sensitivity of the method to noise serves as a sobering reminder of its genesis in the inverse Laplace transform. 

\appendix

\renewcommand{\theequation}{\Alph{section}.\hindu{equation}}

\section{Fourier-based estimation of the \texorpdfstring{$\lambda$}{lambda}'s} \label{sec:trigkernel}
\hspace*{1.4em} In this section, we summarize the approach to obtaining the values of $\lambda_k$ from the Fourier transform $\mathfrak{F}(\tilde{\mu})(\omega)$ as in \eqref{eq:fundarelation}. 

We consider a measure of the form
$$
\mu'=\sum_{k=1}^K b_k\delta_{\lambda_k}.
$$
When we sample $\mathfrak{F}(\mu')(\omega)$ at $\omega=\delta\ell$, $|\ell|<N$ for some $\delta$ and write $\omega_k=\lambda_k\delta$, we obtain the Fourier coefficients of a periodic measure $\mu$: 
\begin{equation}\label{eq:measurefour}
    \hat{\mu}(j)=\sum_{k=1}^K b_k\exp(-ij\omega_k), \qquad j=-N+1,\cdots, N-1.
\end{equation}
Of course, it is assumed that the points $\omega_k \in [0,2\pi]$.
In this section $\|\cdot\|$ will denote $|(\cdot)\mbox{ mod } 2\pi|$.
The material in this section is based on \cite{singdet,trigwave, loctrigwave}.

Clearly, the periodic measure $\mu$ can be recovered formally by its Fourier series, and one is tempted to use the partial sum to take into account that only finitely many Fourier coefficients are known. 
However, a judiciously filtered Fourier sum avoids sidelobes and yields greater accuracy. 
Accordingly, we introduce such a sum (see \eqref{eq:measuresigma}) using a kernel.
Let $h: \RR\to [0,1]$ be an infinitely differentiable, even function, such that $h(t)=1/2$ for $|t|\le 1/2$ and $h(t)=0$ for $|t|\ge 1$. We introduce the kernel 
\begin{equation}
   \label{eq:trigkernel}
\Phi_N(t)=\sum_{\ell\in\ZZ}h\left(\frac{|\ell|}{N}\right)\exp(i\ell t)= \sum_{|\ell|<N}h\left(\frac{|\ell|}{N}\right)\exp(i\ell t), \qquad t\in\RR, \ N>0.
\end{equation}
It is not difficult to see that 
\begin{equation}\label{eq:measuresigma}
   \sigma_N(\mu)(x)=\sum_{|\ell|<N} h\left(\frac{|\ell|}{N}\right)\hat{\mu}(\ell)\exp(i\ell x) =\sum_{k=1}^K a_k\Phi_N(x-\omega_k).
\end{equation}

Now, the kernel $\Phi_N$ has the localization property that 
\begin{equation}\label{eq:kernloc}
    |\Phi_N(t)|\lesssim \frac{\Phi_N(0)}{\max(1, (N\|t \|)^S)}, \qquad N\ge 1, \ t\in\RR.
\end{equation}

If $\eta=\min_{j\not=k}\|(\omega_j-\omega_k\|$, then it is clear that if $N\eta\ge 4$ and  $\|x-\omega_k\|\ge \eta/4$ for all $k$, then
\begin{equation}\label{eq:sigmafar}
    |\sigma_N(\mu)(x)|\lesssim \frac{\Phi_N(0)}{(N\eta)^S}\sum_{k=1}^K |a_k|.
\end{equation}
On the other hand, for any $x$ which does not satisfy the above condition, there is a unique $j$  such that $\|x-\omega_j\|<\eta/4$. 
We have 
$$
|\sigma_N(\omega_j)|\ge \Phi_N(0)|a_j|-c\frac{\Phi_N(0)}{(N\eta)^S}\sum_{k=1}^K |a_k|
$$
for a suitable constant $c>0$.
Using the Bernstein inequality, a concentration inequality related to Hoeffding's inequality that incorporates the variance of the summed random variables, for the trigonometric polynomials, it follows that for sufficiently large values of $N$, the prominent peaks of $|\sigma_N(x)|$ appear within $c_1/N$ of the values of the $\omega_j$'s.
We refer the reader to the above cited literature for further details.

\bibliographystyle{abbrv}
\clearpage
\bibliography{FT_bib_Manuscript}
\end{document}